\documentclass[11pt,a4paper]{amsart}

\usepackage{amsthm}
\usepackage{mathrsfs}
\usepackage{amsfonts,amssymb,amsmath}

\usepackage[dvips]{graphicx}
\usepackage{hyperref}
\usepackage{thmtools}

\usepackage{enumerate}

\declaretheoremstyle[
bodyfont=\normalfont,
]{remstyle}

\newtheorem{lemma}{\bf Lemma}

\newtheorem{theorem}{\bf Theorem}

\newtheorem{corr}{\bf Corollary}

\newtheorem{remark}{\bf Remark}

\newcommand\be{\begin{eqnarray*}}
\newcommand\ee{\end{eqnarray*}}
\newcommand\beq{\begin{equation}}
\newcommand\eeq{\end{equation}}
\newcommand\eps{\epsilon}

\newcommand\ben{\begin{eqnarray}}
\newcommand\een{\end{eqnarray}}

\begin{document}

\date{\today}

\title[On additive bases of sets with small product set]
{On additive bases of sets with small product set}

\author{Ilya D. Shkredov}
\address[1]{Steklov Mathematical Institute
ul. Gubkina, 8, Moscow, Russia, 119991
\and
IITP RAS
Bolshoy Karetny per. 19, Moscow, Russia, 127994}
\email{ilya.shkredov@gmail.com}

\author{Dmitrii Zhelezov}
\address[2]{Mathematical Sciences,
  Chalmers University of Technology,
  SE-412 96 Gothenburg, Sweden
  \and
  Mathematical Sciences,
  University of Gothenburg,
  SE-412 96 Gothenburg, Sweden}
\email{zhelezov@chalmers.se}



\subjclass[2000]{11B13 (primary), 05B10} \keywords{ additive basis, sum-product, additive irreducibility}

\date{\today}

\maketitle

\begin{abstract}
	We prove that finite sets of real numbers satisfying $|AA| \leq |A|^{1+\epsilon}$ with sufficiently small $\epsilon > 0$ cannot have small additive bases nor can they be written as a set of sums $B+C$ with $|B|, |C| \geq 2$. The result can be seen as a real analog of the conjecture of S\'ark\"ozy that multiplicative subgroups of finite fields of prime order are additively irreducible.
\end{abstract}

\maketitle

\section{Introduction}
\label{sec:introduction}

    The duality between the additive and multiplicative structure of an arithmetic set, now called `sum-product phenomena' has been extensively studied since the seminal paper of Erd\H{o}s and Szemer\'edi \cite{ErdosSzemeredi}. In its classical formulation the duality is expressed in the fact that for a finite set of integers or reals $A$ either the set of pairwise sums $A+A$ or pairwise products $AA$ is significantly larger then the original set, unless $A$ is close to a subring of the ambient ring, see \cite{TaoVu} for details.

   In the current paper we consider an intrinsic version of the sum-product phenomenon in the following setting. Assume that a set $A \subset \mathbb{R}$ has small multiplicative doubling, so that $|AA| \leq |A|^{1+\epsilon}$ for some sufficiently small $\epsilon > 0$ (we assume that $\epsilon$ is fixed and $|A| > C(\epsilon)$ is large).

   Recall that a set $B$ is called a \emph{basis}\footnote{All bases in the present paper are assumed to be of order two.} for $A$ if each element in $A$ can be represented as a sum of two elements in $B$ or simply $A \subseteq B+B$.

   First, we show that if $A$ is multiplicatively structured in the above sense, then it does not admit small additive bases of order $|A|^{1/2+c}$ for an explicit constant $c > 0$.
\begin{theorem} \label{thm:smalldoubling}
	There is an $\epsilon > 0$ such that the following holds: if $B$ is an arbitrary additive basis for a real set $A$ with $|AA| \leq |A|^{1+\epsilon}$ then $$
		|B| \gg  |A|^{1/2 + 1/442 - o(1)}.
	$$
\end{theorem}

Theorem \ref{thm:smalldoubling} implies a `power-saving' estimate for the multiplicative energy $E_\times$ of a sumset or a difference set, significantly improving on the sub-exponential bound of \cite{RocheNewtonZhelezov}.
\begin{corr} \label{corr:energy}
 There is an effective constant $\epsilon_0 > 0$ such that for any set $B$  of real numbers holds
 $$
 E_{\times} (B \pm B) \leq |B|^{6 - \epsilon_0}.
 $$
\end{corr}

In fact, we prove an even stronger statement than that of Theorem \ref{thm:smalldoubling}, replacing the condition $A \subseteq B + B$ by a weaker assumption that the number of pairs
$(b_1, b_2) \in B \times B : b_1 + b_2 \in A$ is large, see the upcoming Lemma \ref{lm:convolutions} for details. While slightly more technical, such a reformulation may be useful if one wishes to take into account the number of ways an element in $A$ can be represented as a sum $b_1 + b_2$, see Remark \ref{r:sigma}.

   Second, we show that $A$ is \emph{additively irreducible}, i.e. it cannot be written as a set of sums $B+C$ unless one of the sets consists of a single element.
\begin{theorem} \label{thm:irreducibility}
	There is an $\epsilon > 0$ such that  for all sufficiently large $A \subset \mathbb{R}$ with $|AA| \leq |A|^{1+\epsilon}$ there is no decomposition $A = B + C$ with $|B|, |C| \geq 2$.
\end{theorem}

   Both theorems substantially extend the results of Roche-Newton and the second author \cite{RocheNewtonZhelezov}, where a much more restrictive condition $|AA| \leq K|A|$ with $K$ fixed, was assumed. In particular, in the case at hand deep structural results such as the Freiman-type theorem of Sanders (see e.g. \cite{SandersExp}) and the Subspace theorem of Evertse, Schlickewei and Schmidt \cite{evertse2002linear} are no longer available. Instead, we rely solely on techniques of additive combinatorics, which we see as the main innovation of the present paper. On the other hand, Theorem \ref{thm:irreducibility} extends the result of the first author \cite{ShkredovSmallProductIrreducibility}, which is the specialisation of Theorem \ref{thm:irreducibility} to the case $C = -B$.

Clearly, Theorem \ref{thm:irreducibility} has the following sum-product result as a corollary.

\begin{corr} \label{corr:sum-product}
    There is an absolute constant $c > 0$ such that for any real sets $B, C$ with $|B|, |C| \geq 2$ holds
    $$
        \left| (B+C)(B+C) \right| \gg |B+C|^{1+c}.
    $$
\end{corr}

   Both questions considered in the paper arise as a natural extension of classical problems concerning thin additive bases and additive reducibility of integer sequences to the finite setting.

   One of the basic questions, posed by Erd\H{o}s and Newman \cite{ErdosNewman} in 1976, is to estimate the minimal basis size for a given set $A$ and, in particular, to decide if there is a very thin basis for $A$ of size of order $|A|^{1/2}$. They noted that while for a randomly picked set $A$ one should expect that no such basis exists, it is in general hard to prove such a claim for a specific set $A$.

   The question of additive irreducibility of integer sequences was posed by Ostmann \cite{Ostmann} back in 1956. Perhaps the most notable conjecture of Ostmann, still wide open, is that the set of prime numbers cannot be written as a non-trivial set of sums $B+C$ even if one allows to discard a finite number of elements. We refer the reader to  \cite{elsholtz2006additive} for the history of the problem and \cite{ElsholtzHarper} for the state of the art partial results (see also \cite{GreenHarper}). 
   S\'ark\"ozy \cite{SarkozyDecompositions} (see also \cite{gyarmati2013reducible} and references therein) extended such problems to the finite field setting, perhaps led by the intuition that in general `multiplicatively structured' sets should be additively irreducible (though this intuition may sometimes be misleading, as showed by Elsholtz in \cite{Elsholtz2008}). As a special case of this program, S\'ark\"ozy conjectured that  multiplicative subgroups of finite fields of prime order are additively irreducible, in particular the set of quadratic residues modulo a prime. Despite some progress (see e.g. \cite{ShkredovQuadraticResidues} and references therein), the conjecture of S\'ark\"ozy is an important open problem with rich connections (to e.g. multiplicative character sums).

   From this point of view, Theorem \ref{thm:irreducibility} can be seen as a real analog of  S\'ark\"ozy's conjecture for small subgroups, and in fact the proof can be transferred to the finite field setting except a certain sum-product-type estimate not currently available in finite fields. We discuss this matter in more detail at the end of the paper.

The notation is briefly explained in the next section and is relatively standard in additive combinatorics. We recommend the reader to consult \cite{TaoVu} for further details when needed.

\section{Notation}
\label{sec:notation}

   The following notation is used throughout the paper. The expressions $X \gg Y$, $Y \ll X$, $Y = O(X)$, $X = \Omega(Y)$ all have the same meaning that there is an absolute constant $c > 0$ such that $|Y| \leq c|X|$. The expressions $X \gtrsim Y$ and $Y \lesssim X$ both mean that $|Y| = O(|X| \log^C |X|)$ for some absolute constant $C > 0$.

For a graph $G$, $E(G)$ denotes the set of edges and $V(G)$ denotes the set of vertices. If $X$ is a set then $|X|$ denotes its cardinality.

For sets of numbers $A$ and $B$ the sumset $A + A$ is the set of all pairwise sums  $\{ a + a' : a, a' \in A \}$, and similarly $AA$, $A-A$ denotes the set of products and differences, respectively. If $G$ is some graph with the vertex set identified with $A$ then $A\stackrel{G}{+}A$ denotes the set of sums $\{ a + a' : (a, a') \in E(G) \}$ (with the obvious generalisation to all arithmetic operations).

For a bipartite graph $G$ with parts $(V_1, V_2)$ we will occasionally use the term \emph{bipartite density} for the quantity $\frac{|E(G)|}{|V_1||V_2|}$. For a general graph, the edge \emph{density} is defined as the quotient $\frac{|E(G)|}{|V(G)|^2}$, where the standard notation $E(G), V(G)$ is used for the set of edges and vertices, respectively. For a vertex $v \in V(G)$, $N(v)$ denotes the set of neighbours of $v$.

The additive energy $E_+(A)$  (see \cite{TaoVu}, Definition 2.8) denotes the number of additive  quadruples $(a_1, a_2, a_3, a_4)$ such that $a_1 + a_2 = a_3 + a_4$. The multiplicative energy $E_\times$ is defined similarly as the number of multiplicative quadruples.

\section{Proof of Theorem \ref{thm:smalldoubling}}
\label{sec:proof_smd}

In what follows we will always assume that $0 \notin A$ but $1 \in A$ which one can do without loss of generality.

We will use the following result of the first author (Theorem 5.4 in \cite{shkredov2013}).
\begin{theorem} \label{thm:ShkredovEnergyBound}
	Let $A \subset \mathbb{R}$ be such that $|AA| =  M|A|$. Then
	$$
	E_{+}(A) \ll M^{14/13}|A|^{32/13} \log^{71/65} |A|.
	$$
\end{theorem}

Combining with the Pl\"unnecke-Ruzsa inequality (\cite{TaoVu}, Corollary 6.29), we have the following corollary.
\begin{corr} \label{corr:energy_bound_upper}
	For any $c < 1/26$ there is $\eps > 0$ such that for any $A \subset \mathbb{R}$ with $|AA| \leq |A|^{1+\eps}$ holds
	$$
		E_+(AA/A) \ll |A|^{5/2 - c}.	
	$$
\end{corr}

Recall the following graph--theoretic lemma of Gowers (see e.g. \cite{TaoVu}, Lemma 6.19).
\begin{lemma} \label{lm:pairs}
Let $G$ be a bipartite graph on  $(B_1, B_2)$ where $|B_1| = |B_2| = n$ and $0 < \alpha = E(B_1, B_2)/n^2$. Let $0 < \epsilon < 1$ be fixed. Then there is a subset $B'_1 \subseteq B_1$ with $|B'_1| \geq \alpha n/2$ such that for at least $(1-\epsilon)|B'_1|^2$ of the ordered pairs of vertices $(v_1, v_2) \in B'_1 \times B'_1$ the following holds:
$$
|N(v_1) \cap N(v_2)| \geq \frac{\epsilon \alpha^2 n}{2}.
$$
\end{lemma}

In what follows, we use a weaker definition of a basis, merely assuming that the number of pairs $(b_1, b_2) \in B \times B: b_1 + b_2 \in A$ is large. First, it allows one to take into account `multiplicities', i.e. elements in $A$ which have many representations as a sum. Second, it gives additional flexibility by relaxing the condition that $B+B$ contains \emph{all} elements in $A$.

For this matter, let us call a set $B$ an $(L, K)$\emph{-basis} for $A$, with $K,L \ge 1$, if $|B| = K|A|^{1/2}$ and
$$
	\left| \{(b_1, b_2) : b_1 + b_2  \in A \} \right| = L^{-1}|A|.
$$

We will also make a use of the \emph{containment graph} $G$, which is a bipartite graph $G$ on $(B, B)$ such that $(b_1, b_2) \in E(G)$ if and only if $b_1 + b_2 \in A$.  In particular, if $B$ is an $(L, K)$-basis, one has $|E(G)| = L^{-1}|A| = |B|^2/LK^2$, so the bipartite edge density of $G$ is $L^{-1}K^{-2}$.

\begin{lemma}[The set of popular differences contains an almost closed difference set] \label{lm:convolutions}
	Let $B$ be an $(L, K)$-basis for $A$. Then there is $B' \subset B$ with $|B'| \gg L^{-1}K^{-2}|B|$ such that for at least $0.99 |B'|^2$ of the ordered pairs $(b_1, b_2) \in B' \times B'$ the equation
	\beq \label{eq:representations}
		b_1 - b_2 = a - a'	
	\eeq
	has at least $\Omega(L^{-2}K^{-3}|A|^{1/2})$ solutions $(a, a') \in A \times A$.
\end{lemma}
\begin{proof}
	Applying Lemma \ref{lm:pairs} to the containment graph $G$ we obtain a set $B'$ of size $\Omega(L^{-1}K^{-2}|B|)$ such that for $0.99 |B'|^2$ of the  ordered pairs $(b_1, b_2) \in B' \times B'$ holds
	$$
		|N(b_1) \cap N(b_2)| \geq \Omega(L^{-2}K^{-4}|B|).	
	$$
On the other side, if $b \in N(b_1) \cap N(b_2)$ then by construction $b + b_1 = a$ and $b + b_2 = a'$ for some
$a, a' \in A$, which after rearranging gives
$$
	b_1 - b_2 = a - a'.
$$
Since,  $b + b_1$ are all distinct as $b$ varies, we obtain at least $|N(b_1) \cap N(b_2)|$ distinct solutions of (\ref{eq:representations}) for a fixed pair $(b_1, b_2)$ (and $(b_2, b_1)$ as well).
The result follows.
\end{proof}

Another ingredient is a beautiful incidence result, originally proved by Jones \cite{Jones} with a shorter proof discovered by Roche-Newton \cite{roche2015short}.
\begin{lemma} \label{thm:RN}
	Let $A \subset \mathbb{R}$. Then the number of solutions to
	\beq \label{eq:triangles_collisions}
		 (a-b)(a'-c') = (a-c)(a'-b')	
	\eeq
	such that $a, a', b, b', c, c' \in A$ is $O(|A|^4 \log |A|)$.
\end{lemma}

The following lemma is crucial for the proof.

\begin{lemma}[Generating a large set of popular differences] \label{lm:popular_diff}
    Let $B$ be an $(L, K)$-basis for $A$. Then there is a set $R$ such that the following holds.
    \begin{enumerate}
    		\item[(i)] $R \subseteq A/A$.
    		\item[(ii)] $|R| \gtrsim L^{-8}K^{-14}|A|$.
    		\item[(iii)] For any $x \in R$, the equation
    		$$
    		1-x = \alpha_1 - \alpha_2
    		$$
    		has at least $\Omega(L^{-2}K^{-3}|A|^{1/2})$ distinct solutions $(\alpha_1, \alpha_2)$ with $\alpha_1, \alpha_2 \in A/A$.
    \end{enumerate}
\end{lemma}
\begin{proof}
 Let $B'$ be the set given by Lemma \ref{lm:convolutions}. Let us call a pair $(b_1, b_2) \in B' \times B'$ \emph{rich} if $|N(b_1) \cap N(b_2)| \gg L^{-2}K^{-3}|A|^{1/2}$ in the containment graph $G$, so that at least 99\% of pairs are rich. Take
 $$
		R = \left \{ \frac{b_2 + b}{b_1 + b} : (b_1, b_2) \text{ is rich}, b \in N(b_1) \cap N(b_2) \right\}.
 $$

The first claim follows from the construction of the containment graph, so it remains to prove (ii) and (iii).

For $x \in R$, define
$$
n(x) := \left| \left\{(b_2, b_1, b) : x = \frac{b_2 + b}{b_1 + b}, (b_1, b_2) \text{ is rich}, b \in N(b_1) \cap N(b_2) \right\} \right|.
$$
and
$$
Q := \left| \left\{ (b_2, b_1, b, b'_2, b'_1, b')\in B^6 : \frac{b_2 + b}{b_1 + b} = \frac{b'_2 + b'}{b'_1 + b'} \right\} \right|.
$$

We get  by the Cauchy-Schwartz inequality
$$
	L^{-2}K^{-3}|B'|^2|A|^{1/2} \ll \sum_{x \in R} n(x) \leq |R|^{1/2} \left(\sum_{x \in R} n^2(x) \right)^{1/2} \leq |R|^{1/2}Q^{1/2}.
$$

But by Lemma \ref{thm:RN} we have $Q \ll |B|^4 \log |B|$ and combining with the bound $|B'| \gg L^{-1}K^{-1}|A|^{1/2}$ we obtain
$$
		L^{-8}K^{-14}|A| \lesssim |R|.
$$

For the third bullet, for an $x \in R$ fix $(b_1, b_2, b)$ such that $x =  (b_2 + b)/(b_1 + b)$ and observe that
\begin{equation}\label{f:1-x}
	1 - \frac{b_2 + b}{b_1 + b} = \frac {b_1 - b_2}{b_1 + b} =  \frac {b_1+ b'}{b_1 + b} - \frac{b_2 + b'}{b_1 + b}.
\end{equation}

If we vary over $b' \in N(b_1) \cap N(b_2)$, the pairs   $(b_1 + b')/(b_1 + b), (b_2 + b')/(b_1 + b))$ are all distinct, and the claim follows since
$b_1 + b', b_2 + b', b_1 + b$ are all in $A$ by construction.

\end{proof}

The next lemma is simple but powerful.
\begin{lemma}[Generating a larger set additive quadruples] \label{lm:more_popular_diff}
	Assume $1 \in X$. Let $R \subset X$ such that for any $x \in R$ the equation
	$$
		1 - x = \alpha_1 - \alpha_2	
	$$
	has at least $N$ solutions $(\alpha_1, \alpha_2) \in X \times X$. Then for any set $Y$ one has the estimate
	$$
	E_+(YX) \geq N|Y||R|.
	$$
\end{lemma}
\begin{proof}
	By definition, $E_+(YX)$ is the number of additive quadruples $(y_1, y_2, y_3, y_4)$ such that $y_i \in YX$ and $y_1 - y_2 = y_3 - y_4$.
	If $1 -  x = \alpha_1 - \alpha_2$ with $x, \alpha_1, \alpha_2 \in X$, then clearly
	$(y, yx, y\alpha_1, y\alpha_2)$ is an additive quadruple with elements in $YX$ (remember that $1 \in X$ so $y \in YX$). It remains to check that such quadruples are all distinct.
	But if
	$$
	(y, yx, y\alpha_1, y\alpha_2) = (y', y'x', y'\alpha'_1, y'\alpha'_2)
	$$
	then $y = y'$ and $(x, \alpha_1, \alpha_2) = (x', \alpha'_1, \alpha'_2)$, therefore the number of additive quadruples is at least $|Y|$ times the number of distinct triples
	$(x, \alpha_1, \alpha_2)$, which is at least
	$N|R|$.
\end{proof}

It remains to put everything together.
\begin{proof}{(of Theorem \ref{thm:smalldoubling})}
	Let $B$ be an $(L, K)$-basis. Applying Lemma \ref{lm:popular_diff} and then Lemma \ref{lm:more_popular_diff} with $Y = A$ and $X = R$ from Lemma \ref{lm:popular_diff} we obtain
	\begin{equation} \label{eq:LKenergybound}
		E_+(AA/A) \gg L^{-2}K^{-3}|A|^{3/2}|R| \gtrsim  L^{-10}K^{-17}|A|^{5/2}.
   \end{equation}
On the other hand, by Corollary \ref{corr:energy_bound_upper}
$$
 E_+(AA/A) \lesssim |A|^{5/2 - c}
$$
provided $|AA|/|A|$ is small enough. We then have
\begin{equation} \label{eq:LKbound}
	L^{10}K^{17} \gtrsim |A|^{c}.
\end{equation}
In particular, if $B$ is a basis for $A$ then
$$
	|B| \gtrsim |A|^{1/2 + c/17}.
$$	
	
\end{proof}

We record an immediate Corollary of Theorem \ref{thm:smalldoubling}.

\begin{corr}
There is an  $\epsilon > 0$ and an absolute constant $c > 0$, such that for any $A \subset \mathbb{R}$ with $|AA| \leq |A|^{1+\epsilon}$ the following holds.
If  $A \subset B+C$ for some real sets $B, C$ then
$$
	\max(|B|, |C|) \gg |A|^{1/2 + c}.
$$
If $B + C \subset A$ for some real sets $B, C$ then
$$
	\min(|B|, |C|) \ll |A|^{1/2 - c}.
$$
\end{corr}
\begin{proof}
	The first claim follows immediately from Theorem \ref{thm:smalldoubling} and the fact that $B \cup C$ is a basis for $A$ of size $O(\max(|B|, |C|))$. 
	
	In the second case, 
assume $|B| \leq |C|$ and let $B'$ be an arbitrary subset of $C$ of size $|B|$. Then $B \cup B'$ is a basis for $A' := B+B'$ of size at most twice $\min(|B|, |C|)$. Applying Lemma \ref{lm:popular_diff} and then Lemma \ref{lm:more_popular_diff}  we obtain similarly to (\ref{eq:LKenergybound}) 
$$
	E_+(A'A'/A') \gtrsim |B'|^2 |B'| |B'|^2 = |B'|^{5}.
$$
But clearly $A' \subset A$, so $E_+(A'A'/A') \leq E_+(AA/A)$. On the other hand, by Corollary \ref{corr:energy_bound_upper}
$$
 E_+(AA/A) \lesssim |A|^{5/2 - c}
$$
and the claim follows.
\end{proof}

The proof of Corollary \ref{corr:energy} is a simple application of the Balog-Szemer\'edi-Gowers theorem (BSG) (see \cite{TaoVu}) so we present it in a sketchy manner. Let $X = B \pm B$. Since trivially $E_{\times}(X) \ll |X|^3$
we may assume that $|X| \gg |B|^{2 - \epsilon}$ with $\epsilon$ to be defined in due course.

If now one assumes that
$$
E_\times(X) \gg |X|^{3-\epsilon},
$$
then by BSG there is $X' \subset X$ such that $|X'X'| \leq |X'|^{1+\epsilon'}$ and $|X'| \gg |X|^{1-\epsilon'}$ with $\epsilon'$ depending polynomially on $\epsilon$. But  $X' \subset X =  B \pm B$ and whence $B \cup (-B)$ is a basis for $X'$ of size $O(|X'|^{1/2 + \epsilon''})$ with polynomial dependence of $\epsilon''$ on $\epsilon$. Taking $\epsilon$ small enough, we contradict Theorem \ref{thm:smalldoubling} and the claim follows.

\begin{remark}
\label{r:sigma}
    The following question was posed in \cite{ShkredovSmallProductIrreducibility}. Assume that $A\subseteq B-B$, $|AA| \le |A|^{1+\varepsilon}$, where $\varepsilon > 0$ is small. Is it true that there is $\delta= \delta(\varepsilon)$ such that 
$$
    \sigma_A (B) := \sum_{x\in A} |\{ b_1-b_2 = x ~:~ b_1,b_2 \in B \}| \ll |B|^{2-\delta} \quad \mbox{?} 
$$
    It is easy to see that our methods allow one to resolve the problem in the affirmative. Indeed, put $\sigma = \sigma_A (B)$ and consider the graph with $B$ as the vertex set and two vertices $x,y$ adjacent iff $x-y \in A$. Then the edge density of the graph is $\alpha := \sigma / |B|^2$. Next, repeat the arguments of Section \ref{sec:proof_smd} and obtain that for a sufficiently small $\varepsilon$ holds 
$$
    |A|^{5/2-c} \gtrsim \alpha^2 |B| |A| |R| \gtrsim \alpha^{10} |B|^3 |A|
$$
    and hence since $A\subseteq B-B$ 
\begin{equation}\label{tmp:06.06.2016}
    \sigma \lesssim |B|^2 \cdot |A|^{-c/10} \cdot \left( \frac{|A|^{3/2}}{|B|^3} \right)^{1/10}
        \leq 
            |B|^2 \cdot |A|^{-c/10} \,.
\end{equation}
    The asymmetric case with two different sets $B,C$ is considered in the next section. In this case one can obtain an analog of the upper bound (\ref{tmp:06.06.2016}) of the form 
$$
    \sum_{x\in A} |\{ b+c = x ~:~ b\in B,\, c\in C \}| 
        \lesssim 
            |B||C| \cdot |A|^{-c/10} \,. 
$$
\end{remark}

\begin{remark}
    It is highly probable that the exponent $1/2+1/442$ in Theorem \ref{thm:smalldoubling} can be improved, perhaps even to $1-o(1)$. We pose a much more modest question, which however does not seem to follow from the results of the present paper.
    
    Is it true that there exists $c, \epsilon > 0$ such that for all sufficiently large real sets $B$ such that 
    $A \subset B+B $ with $|AA| \leq |A|^{1+\epsilon}$
    holds
    $$
        |B+B| \gg |A|^{1+c} \,\, ?
    $$
\end{remark}

\section{Additive irreducibility of multiplicative sets}
\label{sec:decomposition}

In this section we prove Theorem \ref{thm:irreducibility}.
From now on we assume for the sake of contradiction that $A = B+C$ with $|B| \geq |C| \geq 2$  and $|AA| \leq |A|^{1+\epsilon}$ with $\epsilon$ small enough.

First, note that it follows from Lemma 29 of \cite{ShkredovSmallProductIrreducibility} that for $\alpha \neq 0$ and $A$ such that
$|AA| \leq M|A|$ holds

\begin{equation} \label{eq:convolution_bound}
|A \cap (A + \alpha)| \leq M^{4/3}|A|^{2/3}.
\end{equation}

For any $c_1 \neq c_2 \in C$ one has $(B + c_1) \subseteq A \cap (A + (c_1 - c_2))$, and thus by (\ref{eq:convolution_bound}) and the trivial bound $|B||C| \geq |A|$ one concludes that $|B|, |C| \gg |A|^{1/3 - \epsilon}$, say. Hence, we can safely assume that both $|B|, |C|$ are large.

An inspection of the proof of Theorem \ref{thm:smalldoubling} reveals that if we put

$$
	X = \left\{ \frac{b_1+c}{b_2+c} : b_1, b_2 \in B,  c \in C \right\}
$$
and
$$
	Y = \left\{ \frac{c_1+b}{c_2+b} : b \in B,  c_1, c_2 \in C \right\}
$$
then one has that
$$
  E_+(AA/A) \gg \min \large( |A||X||C|, |A||Y||B| \large).
$$
Thus, by Corollary \ref{corr:energy_bound_upper} one has
$$
	|X||C| \ll |A|^{3/2 - c}
$$
and
\beq \label{eq:energy_C}
	|Y||B| \ll |A|^{3/2 - c}
\eeq
for some explicit $c > 0$.

However, sufficiently good lower bounds for $|X|$ and $|Y|$ are not readily available. 

Let $T(X, Y, Z)$ be the number of collinear triples of distinct points $(x, y, z)$ with $x \in X \times X$, $y \in Y \times Y$, $z \in Z \times Z$. Following the lines of Lemma \ref{lm:popular_diff} we have by the Cauchy--Schwarz inequality
\beq \label{eq:bound_X}
	|X| \gg \frac{|B|^4|C|^2}{T(B, B, -C)} \,,
\eeq
and
\beq \label{eq:bound_Y}
	|Y| \gg \frac {|C|^4|B|^2}{T(C, C, -B)} \,,
\eeq

Indeed, without loss of generality it suffices to check that if, say,
$$
y := \frac{c_1+b}{c_2+b} = \frac{c'_1+b'}{c'_2+b'}
$$
then either the points $(c_1, c'_1), (c_2, c'_2), (-b, -b')$ are all distinct and collinear or $y \in \{0, 1, \infty \}$. But we can simply exclude such degenerate values from $X$ and $Y$ and thus justify (\ref{eq:bound_X}) and (\ref{eq:bound_Y}).

We prove the following estimate with $|C| \leq |B|$ which improves on Lemma \ref{thm:RN} when $|B|$ is significantly larger than $|C|$.

\begin{lemma}[Bounding collinear triples for different sets] \label{lm:T_CCB} For sets $C \times C$ and $B \times B$  with $|B| \geq |C|$ holds
	$$T(C, C, B) \ll |B|^{4/3}|C|^{8/3} \log^2 |B| \,.$$
\end{lemma}
\begin{proof}
    Write  $L_{i, j}$, $i,j\ge 0$ for the set of lines $\ell$ such that
    $2^{i} \leq |\ell \cap (C \times C)| < 2^{i+1}$ and $2^{j} \leq |\ell \cap (B \times B)| < 2^{j+1}$. Then
    $$
		T(C, C, B) \ll \sum^{\log |C|}_{i = 0}\, \sum^{\log |B|}_{j = 0} |L_{i, j}| 2^{2i}2^{j}.
    $$

Since the number of summands is at most $\log^2 |B|$ it is enough to bound each term by $|B|^{4/3}|C|^{8/3}$. For the sake of notation, denote $k = 2^i$ and
$l = 2^j$, $L = L_{i,j}$ so that out task is to estimate $|L|k^2l$ where $L$ is the set of lines intersecting $C \times C$ in $k$ (up to a factor of two) points and $B \times B$ in $l$ points (again, up to a factor of two).

By the Szemer\'edi-Trotter theorem \cite{TaoVu}, for $k \geq 2, l \geq 2$ holds
$$
	|L| \ll \min \left( \frac{|C|^4}{k^3} + \frac{|C|^2}{k}, \frac{|B|^4}{l^3} + \frac{|B|^2}{l} \right),
$$
so
$$
	T := k^2l|L| \ll  \min \left( \frac{ l |C|^4}{k} + kl|C|^2, \frac{k^2|B|^4}{l^2} + k^2|B|^2 \right) := \min (M_1, M_2).
$$

Before we proceed, let us rule out the cases not covered by the Szemer\'edi-Trotter theorem as stated above. Since each line contains at least two distinct points in $C \times C$, only the case
$l = 1$ should be considered separately. In this case we have
$$
T \ll \frac{|C|^4}{k} + k|C|^2 \ll |C|^4 \leq |B|^{4/3}|C|^{8/3},
$$
which is the desired bound.

Next, since always $k \leq |C|$, $l \leq |B|$, we obtain that $\frac{ l |C|^4}{k} \gg M_1$ and $\frac{k^2 |B|^4}{l^2} \gg M_2$, 
so
$$
    T^3 \ll M_1^2 M_2 \ll \left( \frac{ l |C|^4}{k} \right)^2 \cdot \frac{k^2 |B|^4}{l^2} = |C|^8 |B|^4 
$$
or 
$$
    T \ll |B|^{4/3} |C|^{8/3} \,.
$$
This finishes the proof.
\end{proof}

Combining Lemma \ref{lm:T_CCB}, estimate (\ref{eq:bound_Y}) and the trivial bounds $|B||C| \geq |A|, |B| \geq |A|^{1/2}$, we have $T(C, C, B) \lesssim |C|^{8/3}|B|^{4/3} $ and
$$
	|Y||B| \gtrsim |C|^{4/3}|B|^{2/3}|B| = (|C||B|)^{4/3} |B|^{1/3} \geq |A|^{4/3 + 1/6} = |A|^{3/2},
$$
which contradicts (\ref{eq:energy_C}) if $|A|$ is large enough. Thus, $A$ is additively irreducible.

\begin{remark}
    Instead of using the graph approach as it was done in Sections \ref{sec:proof_smd} and \ref{sec:decomposition} one can apply in the proof of Theorem \ref{thm:irreducibility} the following analog of the formula (\ref{f:1-x})
$$
    1-\frac{b_1+c}{b_2+c} = \frac{b_2+c'}{b_2+c} \left( 1 - \frac{b_1+c'}{b_2+c'} \right), 
$$
    which holds for any $b_1,b_2\in B$ and all $c,c'\in C$. It means in particular, in the notation of Section  \ref{sec:decomposition}, that taking an arbitrary $x\in X$, we have at least $|C|$ solutions to the equation 
$$
    1-x = y (1-x_*) \,,
$$
    where $y\in Y$ and $x_* \in X$ are fixed. Of course, one can replace $X$ and $Y$ in the last formula. It immediately gives that $E_{+} (AA/A) \ge |X||C||A|$ and $E_{+} (AA/A) \ge |Y||B||A|$. 
\end{remark}

\section{Discussion}
\label{sec:discussion}
		
	First, let us note that the only feature of the reals we have used, which is not available in an arbitrary field, is the Szemer\'{e}di-Trotter theorem (which is also used in the proof of Theorem \ref{thm:ShkredovEnergyBound}). Thus, our results can be readily extended to sets of complex numbers thanks to the extensions of the Szemer\'{e}di-Trotter theorem to the complex plane due to Zahl \cite{MR3392965} and Toth \cite{toth2015szemeredi}.
	
	However, when we replace $A$ with a sufficiently small multiplicative subgroup of $\mathbb{F}_p$, the situation becomes more subtle. For multiplicative subgroups Theorem \ref{thm:ShkredovEnergyBound} can be substituted with a similar energy bound, beating the exponent $5/2$, as was shown by the first author \cite{Sh_ineq}. Next, it follows from the result of Shkredov and Vyugin \cite{MR2984656} (see also the paper \cite{ShparlinskiGroupDecomposition} by Shparlinski) that if a multiplicative subgroup $G$ is written as a non-trivial sumset $B+C$, then necessarily $|G|^{1/2+o(1)} \ll |B|, |C| \ll |G|^{1/2+o(1)}$. 
Further, a slightly weaker form of the Szemer\'{e}di-Trotter bound for the number of incidences for Cartesian products in $\mathbb{F}^2_p$ is now available, see \cite{AMRS}.
Thus, the only remaining obstacle for translating Theorem \ref{thm:irreducibility} to subgroups is to find a substitute for Lemma \ref{thm:RN}, which we were unable to do.  However, as was shown in \cite{AMRS}, see also \cite{RudnevRocheShkredov} the bound 
$O(|A|^4 \log |A|)$ of Lemma 4 can indeed be replaced with $O(|A|^{9/2})$, which only barely falls short of the desired bound.
	
	Finally, let us mention that similarly to Section \ref{sec:proof_smd} one can slightly weaken the hypothesis of Theorem \ref{thm:irreducibility} and assume only that the number of pairs $(b, c) \in B \times C : b + c \in A$ is large. We leave the details for the interested reader. 
	

\section*{Acknowledgements}
The authors would like to thank Tomasz Schoen, Misha Rudnev and Oliver Roche-Newton for useful discussions and the anonymous referees for valuable suggestions. The second author would like to acknowledge the hospitality of Adam Mickiewicz University in Poznan at which the work on the present paper was partially conducted. 

The work of I. D. Shkredov was supported by Russian Scientific Foundation grant 14-11-00433.

\bibliographystyle{plain}
\bibliography{bases_for_mult_sets_irmn}

\begin{thebibliography}{10}

\bibitem{elsholtz2006additive}
C.~Elsholtz.
\newblock Additive decomposability of multiplicatively defined sets.
\newblock {\em Functiones et Approximatio Commentarii Mathematici},
  35(1):61--77, 2006.

\bibitem{Elsholtz2008}
C.~Elsholtz.
\newblock Multiplicative decomposability of shifted sets.
\newblock {\em Bull. Lond. Math. Soc.}, 40(1):97--107, 2008.

\bibitem{ElsholtzHarper}
C.~{Elsholtz} and A.~{Harper}.
\newblock {Additive decompositions of sets with restricted prime factors}.
\newblock {\em {Trans. Amer. Math. Soc.}}, (367):7403--7427, 2015.

\bibitem{ErdosNewman}
P.~{Erd\H{o}s} and D.~J. Newman.
\newblock {Bases for sets of integers}.
\newblock {\em J. Number Theory}, 9(4):420--425, 1977.

\bibitem{ErdosSzemeredi}
P.~{Erd\H{o}s} and E.~{Szemer\'edi}.
\newblock {\em {Sums and products of integers}}.
\newblock {Studies in Pure Mathematics}. {Birkh\"auser, Basel}, 1983.

\bibitem{evertse2002linear}
J.-H. Evertse, H.~P. Schlickewei, and W.~M. Schmidt.
\newblock Linear equations in variables which lie in a multiplicative group.
\newblock {\em Ann. of Math.}, 155(3):807--836, 2002.

\bibitem{GreenHarper}
B.~Green and A.. Harper.
\newblock Inverse questions for the large sieve.
\newblock {\em Geom. Funct. Anal.}, 24(4):1167--1203, 2014.

\bibitem{gyarmati2013reducible}
K.~Gyarmati, C.~Mauduit, and A.~S{\'a}rk{\"o}zy.
\newblock {On reducible and primitive subsets of $\mathbb{F}_p$, I}.
\newblock {\em INTEGERS}, 2013.

\bibitem{Jones}
T.G.F. Jones.
\newblock New quantitative estimates on the incidence geometry and growth of
  finite sets.
\newblock {\em PhD thesis, arXiv:1301.4853, 2013}.

\bibitem{Ostmann}
H.-H. Ostmann.
\newblock {\em Additive Zahlentheorie, 2 Vols.}
\newblock Springer, Berlin, 1956.

\bibitem{roche2015short}
O.~Roche-Newton.
\newblock A short proof of a near-optimal cardinality estimate for the product
  of a sum set.
\newblock arXiv:1502.05560, 2015.

\bibitem{RudnevRocheShkredov}
O.~{Roche-Newton}, M.~{Rudnev}, and I.~D. {Shkredov}.
\newblock New sum-product type estimates over finite fields, 2016.

\bibitem{RocheNewtonZhelezov}
O.~{Roche-Newton} and D.~{Zhelezov}.
\newblock {A bound on the multiplicative energy of a sum set and extremal
  sum-product problems}.
\newblock {\em {Moscow J. Comb. and Number Theory}}, 5(1):53--70, 2015.

\bibitem{SandersExp}
T.~{Sanders}.
\newblock {The structure theory of set addition revisited}.
\newblock {\em {Bull. Amer. Math. Soc. (N.S.)}}, 50(1):93--127, 2013.

\bibitem{SarkozyDecompositions}
A.~{S\'ark\"ozy}.
\newblock {On additive decompositions of the set of quadratic residues modulo
  $p$}.
\newblock {\em {Acta Arith.}}, (155):41--51, 2012.

\bibitem{Sh_ineq}
I.~D. Shkredov.
\newblock Some new inequalities in additive combinatorics.
\newblock {\em Moscow J. Combin. Number Theory}, (3):237--288, 2013.

\bibitem{shkredov2013}
I.~D. Shkredov.
\newblock Some new results on higher energies.
\newblock {\em Trans. Moscow Math. Soc.}, 74:31--63, 2013.

\bibitem{ShkredovQuadraticResidues}
I.~D. {Shkredov}.
\newblock {Sumsets in quadratic residues}.
\newblock {\em {Acta Arith.}}, (164):221--243, 2014.

\bibitem{ShkredovSmallProductIrreducibility}
I.~D. {Shkredov}.
\newblock {Difference sets are not multiplicatively closed}.
\newblock {\em Discrete Analysis}, (17), 2016.

\bibitem{MR2984656}
I.~D. Shkredov and I.~V. Vyugin.
\newblock On additive shifts of multiplicative subgroups.
\newblock {\em Mat. Sb.}, 203(6):81--100, 2012.

\bibitem{ShparlinskiGroupDecomposition}
I.~{Shparlinski}.
\newblock {Additive decompositions of subgroups of finite fields}.
\newblock {\em {SIAM J. Discrete Math.}}, 27(4):1870--1879, 2013.

\bibitem{TaoVu}
T.~Tao and V.~Vu.
\newblock {\em Additive Combinatorics}.
\newblock Cambride University Press, 2006.

\bibitem{toth2015szemeredi}
C.~T{\'o}th.
\newblock {The Szemer{\'e}di-Trotter theorem in the complex plane}.
\newblock {\em Combinatorica}, 35(1):95--126, 2015.

\bibitem{AMRS}
A.S. Yazici, B.~Murphy, M.~Rudnev, and I.D. Shkredov.
\newblock Growth estimates in positive characteristic via collisions.
\newblock {\em Int. Math. Res. Notices}, 2016:1--42, 2016.

\bibitem{MR3392965}
J.~Zahl.
\newblock A {S}zemer\'edi-{T}rotter type theorem in {$\Bbb{R}\sp 4$}.
\newblock {\em Discrete Comput. Geom.}, 54(3):513--572, 2015.

\end{thebibliography}

\end{document}